\documentclass[12pt,a4paper]{amsart}
\usepackage{a4wide}
\usepackage{amsmath, amssymb, amsfonts,enumerate}
\usepackage[all]{xy}
\usepackage{amscd}
\usepackage{comment}
\usepackage{mathtools}
\usepackage{url}
\usepackage{hyperref}

\newtheorem{theorem}{Theorem}[section]
\newtheorem{lemma}[theorem]{Lemma}
\newtheorem{corollary}[theorem]{Corollary}
\newtheorem{proposition}[theorem]{Proposition}

 \theoremstyle{definition}
 
 \newtheorem{remark}[theorem]{Remark}

 \newtheorem{example}[theorem]{Example}

\numberwithin{equation}{section}
\newcommand {\K}{\mathbb{K}} %% a field
\newcommand {\Z}{\mathbb{Z}} %% integers
\newcommand {\Q}{\mathbb{Q}} %% rationals
\newcommand {\C}{\mathbb{C}} %% complex

\DeclareMathOperator{\CA}{CA}
\DeclareMathOperator{\LCA}{LCA}

\DeclareMathOperator{\End}{End}

\DeclareMathOperator{\Mat}{Mat}

\DeclareMathOperator{\Id}{Id}

\DeclareMathOperator{\supp}{supp}

\begin{document}
\title[Stable finiteness of monoid algebras]{Stable finiteness of monoid algebras and surjunctivity}
\author[T.Ceccherini-Silberstein]{Tullio Ceccherini-Silberstein}
\address{Universit\`a del Sannio, I-82100 Benevento, Italy}
\address{Istituto Nazionale di Alta Matematica ``Francesco Severi'', I-00185 Rome, Italy}
\email{tullio.cs@sbai.uniroma1.it}
\author[M.Coornaert]{Michel Coornaert}
\address{Universit\'e de Strasbourg, CNRS, IRMA UMR 7501, F-67000 Strasbourg, France}
\email{michel.coornaert@math.unistra.fr}
\author[X.K.Phung]{Xuan Kien Phung}
\address{D\'epartement d'Informatique et de Recherche Op\'erationnelle, Universit\'e de Montr\'eal, Montr\'eal, Qu\'ebec, H3T 1J4, Canada}
\address{D\'epartement de Math\'ematiques et de Statistique, Universit\'e de Montr\'eal, Montr\'eal, Qu\'ebec, H3T 1J4, Canada}
\email{phungxuankien1@gmail.com}
% 16S36 Ordinary and skew polynomial rings and semigroup rings
% 20M25 Semigroup rings, multiplicative semigroups of rings
% 20M35 Semigroups in automata theory, linguistics, etc.
% 03C98 Applications of model theory
% 37B15 Cellular automata
% 68Q80  Cellular automata 
\begin{abstract}
A monoid $M$ is said to be surjunctive if every injective cellular automaton with finite alphabet over $M$ is surjective.
We show that monoid algebras of surjunctive monoids are stably finite.
In other words, given any field $K$ and any surjunctive monoid $M$,
every one-sided invertible square matrix with entries in the monoid algebra $K[M]$ is two-sided invertible. 
Our proof uses first-order model theory.
\end{abstract}
\date{\today}
\maketitle

\setcounter{tocdepth}{1}
\tableofcontents

% SECTION 1
\section{Introduction}

A \emph{monoid} is a semigroup with $1$, that is, a set equipped with an associative binary operation that admits an identity element.
One says that a monoid $M$ is \emph{directly finite}  if
$ab = 1$ implies $ba = 1$ for all $a,b \in M$.
This amounts to saying that every one-sided invertible element in $M$ is invertible.
All finite monoids, all commutative monoids, all unit-regular monoids, all one-sided cancellative monoids and in particular all groups   are  directly finite.
On the other hand, the \emph{bicyclic monoid}, i.e., the monoid given by the presentation $B \coloneqq  \langle p,q : pq = 1 \rangle$,
is not directly finite since $pq = 1 \not= qp$.
Actually, a monoid is directly finite if and only if it contains no submonoid isomorphic to the bicyclic monoid
(see e.g.~\cite[Proposition~2.3]{csc-surj-monoids}). 
\par
Let $R$ be a ring (all rings considered in this paper are assumed to be associative with a multiplicative identity element  $1$).
One says that $R$ is \emph{directly finite}  if
the multiplicative monoid underlying $R$ is directly finite.
This amounts to saying   that $R$ is Hopfian as a left (resp.~right) module over itself
(recall that a module $\Lambda$ is called  \emph{Hopfian} if every surjective endomorphism of $\Lambda$ is injective). 
The ring $R$ is said to be \emph{stably finite} if the ring $\Mat_d(R)$ of $d \times d$ square matrices with entries in $R$ is directly finite for every integer $d \geq 1$.
A ring $R$ is stably finite if and only if every finitely generated free left (resp.~right) $R$-module is Hopfian.
All finite rings, all commutative rings, all fields, all division rings, all one-sided Noetherian rings, all one-sided Artinian rings,
all unit-regular rings, and all finite-dimensional algebras over fields are stably finite
(see~\cite{shepherdson}, \cite[Section~1.B]{lam-modules-rings}, \cite{goodearl-vnr-rings}, \cite[Chapter~8]{ca-and-groups-springer}, \cite[Chapter~8]{csc-exos}).
On the other hand, the endomorphism ring $R \coloneqq \End_K(V)$ of an infinite-dimensional vector space $V$ over a field $K$ is not directly finite.
The class of directly (resp.~stably) finite rings is closed under taking subrings, direct products, direct sums, direct limits, and inverse limits.
Any stably finite ring is directly finite  since $\Mat_1(R) = R$ for any ring $R$.
The converse is false in general:  there exist rings that are directly finite but not stably finite
(see~\cite{jacobson-1950}, \cite{shepherdson}, \cite{cohn-remarks-ibn}, \cite[Exercise~1.18]{lam-emr}).
In fact, for any integer $d \geq 1$ and any field $K$, one can construct a $K$-algebra $R$ such that
$\Mat_d(R)$ is directly finite but $\Mat_{d + 1}(R)$ is not (see \cite[Theorem~7.2]{cohn-remarks-ibn}).
\par
Given a field $K$ and a monoid $M$, the \emph{monoid algebra} $K[M]$ is  
 the  $K$-algebra  with vector space basis $M$
with the multiplicative operation obtained by extending $K$-linearly the monoid operation on $M$.
\par
Let $M$ be a monoid and let $A$ be a set.
Consider the set  $A^M$ consisting of all maps $c \colon M \to A$.
For $S \subset M$ and $c \in A^M$, denote by $c|_S$ the restriction of $c$ to $S$, i.e., the element $c|_S \in A^S$ given by
$c|_S(s) \coloneqq c(s)$ for all $s \in S$.
For $m \in M$, let $R_m \colon M \to M$ denote the right multiplication by $m$, i.e., the map 
given by $R_m(x) \coloneqq x m$ for all $x \in M$.
One says that a map $\tau \colon A^M \to A^M$ is a \emph{cellular automaton} over $(M,A)$ 
if there exists a finite subset $S \subset M$ and a map $\mu \colon A^S \to A$ 
such that
\begin{equation}
\label{e:cell-aut}
\tau(c)(m) = \mu\left((c \circ R_m)\vert_S\right)
\end{equation}
for all $c \in A^M$ and $m \in M$.
\par
A monoid $M$ is called \emph{surjunctive} if, for any finite set $A$, every injective  cellular automaton over $(M,A)$   is surjective.
In the particular case when $M$ is a group, this  definition of surjunctivity is equivalent to the one given by
Gottschalk~\cite{gottschalk} (see also~\cite{gromov-esav},  \cite[Chapter~3]{ca-and-groups-springer}, \cite{csc-surj-monoids}).
In~\cite{csc-surj-monoids}, it is shown that all finite monoids, all finitely generated commutative monoids, all cancellative commutative monoids, all residually finite monoids, all finitely generated linear monoids, 
and all cancellative one-sided amenable monoids are surjunctive.
Every surjunctive monoid is directly finite
\cite[Theorem~5.5]{csc-surj-monoids}.
Thus, every  monoid containing a submonoid isomorphic to the bicyclic monoid is non-surjunctive.
In particular, the bicyclic monoid itself is non-surjunctive.
By contrast, no example of a non-surjunctive group has been found up to now.
The conjecture that every group is surjunctive is known as \emph{Gottschalk's surjunctivity conjecture}.
\par
The goal of the present paper is to establish the following result.

\begin{theorem}
\label{t:main}
Let $M$ be a surjunctive monoid and let $K$ be a field. 
Then the monoid algebra $K[M]$ is stably finite.
\end{theorem}

Given a field $K$, any monoid $M$ embeds into the multiplicative monoid $K[M]$ by definition.
Consequently, $K[M]$ is not directly finite  as soon as   $M$ is not directly finite.
In particular, the monoid algebra $K[B]$ of the bicyclic monoid $B$ or, more generally, the monoid algebra $K[M]$, where $M$ is a monoid containing a submonoid isomorphic to $B$,  is not directly finite whatever the field $K$.
\par
A monoid algebra over a group is called a \emph{group algebra}.
One of the celebrated conjectures attributed to Kaplansky about the algebraic structure of group algebras is that every group algebra is stably finite.
This conjecture is known as \emph{Kaplansky's stable finiteness conjecture}.
Kaplansky~\cite[p.~122]{Kap2}, \cite[Problem~23]{Kap1} proved this conjecture  when the ground field $K$ is of characteristic $0$,
but  the case when $K$ is of positive characteristic remains open (see~\cite{steinberg2022stable} for an extension of Kaplansky's result to certain classes of monoids). 
The fact that surjunctive groups satisfy Kaplansky's stable finiteness conjecture is a particular case of Theorem~\ref{t:main}
that  was first established by the third author in~\cite[Theorem~B]{phung-pams-2023}  using techniques from algebraic geometry.
Direct finiteness of group algebras of surjunctive groups was proved by Capraro and Lupini in~\cite[p.~10]{capraro-lupini}  
 and their stable finiteness  can be then deduced by applying a result of Dykema and Juschenko~\cite[Theorem~2.2]{dykema-juschenko}
(see~\cite[Remark~4.3]{cscp-model}).
Another proof of stable finiteness of group algebras of surjunctive groups,
based on completeness of the first-order model theory of algebraically closed fields,
can be found in~\cite{cscp-model}.
The proof of Theorem~\ref{t:main} that we present in this paper follows the lines of the one given in~\cite{cscp-model}.
\par
The paper is organized as follows.
In Section~\ref{s:background}, we introduce some notation and collect some background material on
configuration spaces, cellular automata, and model theory. 
In Section~\ref{s:lca}, we investigate  linear cellular automata over monoids
and extend some of the results presented in~\cite[Chapter~8]{ca-and-groups-springer} for groups.
We show in particular (cf.~Theorem~\ref{t:rep-d-dim-lca}) that, given a monoid $M$, a field $K$, and an integer $d \geq 1$, 
the $K$-algebra of all linear cellular automata $\tau \colon (K^d)^M \to (K^d)^M$ is canonically anti-isomorphic to $\Mat_d(K[M])$. 
The proof of Theorem~\ref{t:main} is given in Section~\ref{s:proof-main-result}.
The final section contains a list of some open problems.

% SECTION 2
\section{Background material}
\label{s:background}

\subsection{Sets and maps}
Given sets $A$ and $B$, 
we denote by $B^A$ the set consisting of all maps $f \colon A \to B$.
Given sets $A$, $B$,  $C$, and maps $f \in B^A$ and $g \in C^B$, we denote by $g \circ f$ their \emph{composite}, i.e., the map $g \circ f \in C^A$ defined by $(g \circ f)(a) = g(f(a))$ for all $a \in A$.
If $C \subset A$ and $f \in B^A$,
we denote by $f|_C$ the \emph{restriction} of $f$ to $C$, i.e., the element $f|_C \in B^C$ defined by
$f|_C(c) \coloneqq f(c)$ for all $c \in C$.
When $A$ is a finite set, we write $|A|$ for its cardinality.

\subsection{Monoids}
Recall that a \emph{monoid} is a set equipped with a binary associative operation that admits an identity element. 
Let $M$ be a monoid.
Given an element $m \in M$, we denote by $L_m$ and $R_m$   the left and right multiplication by $m$, that is,
the maps $L_m \colon M \to M$ and $R_m \colon M \to M$ defined by $L_m(m') \coloneqq  mm'$ and $R_m(m') \coloneqq  m'm$, respectively,  for all $m' \in M$.
An element $m \in M$  is called \emph{left-cancellable} (resp.~\emph{right-cancellable}) 
if the map $L_m$ (resp.~$R_m$) is injective.
One says that an element $m \in M$ is cancellable if it is both left-cancellable and right-cancellable.
The monoid $M$ is called \emph{left-cancellative} (resp.~\emph{right-cancellative}, resp.~\emph{cancellative}) if every element in $M$ is left-cancellable  (resp.~right-cancellable, 
resp.~cancellable).
An element $m \in M$ is called
\emph{right-invertible} (resp.~\emph{left-invertible}) if  the map  $L_m$ (resp.~$R_m$)
is surjective.
This amounts to saying that there exists an element $m' \in M$ such that $m m' = 1_M$
(resp.~$m' m = 1_M$).
Such an element $m'$ is then called a \emph{right-inverse}
(resp.~\emph{left-inverse}) of $m$.
An element  $m \in M$ is called \emph{invertible} if it is both right-invertible  and left-invertible.
This amounts to saying that there exists an element $m' \in M$ such that $m m' = m' m = 1_M$.
The element $m'$ is then unique and is called the \emph{inverse} of $m$. 
It is both the unique right-inverse and the unique left-inverse of $m$.
    A \emph{group} is a monoid in which every element is invertible.
\par
If $M$ and $M'$ are monoids, a \emph{monoid homomorphism} from $M$ into $M'$ is a map $f \colon M \to M'$ such that 
$f(1_M) = 1_{M'}$ and $f(m_1 m_2) = f(m_1)f(m_2)$ for all $m_1,m_2 \in M$.
\par
A \emph{submonoid} of a monoid $M$ is a subset $N \subset M$ such that $1_M \in N$ and $m m' \in N$ for all $m, m' \in N$.
If $N \subset M$ is a submonoid, then $N$ inherits from $M$ a monoid structure obtained by restricting to $N$ the monoid operation on $M$.
The inclusion map $N \to M$ is then a   monoid homomorphism.

\subsection{Monoid actions}
An \emph{action} of a monoid $M$ on a set $X$ is a map $M \times X \to X$, $(m,x) \mapsto m x$,  such that
$1_M x = x$ and $m_1 (m_2 x) = (m_1 m_2) x$ for all $x \in X$ and $m_1,m_2 \in X$.
Observe that the map $M \times M \to M$, given by $(m,x) \mapsto x m$ for all $(m,x) \in M \times M$, is an action of $M$ on itself.
\par
A set equipped with an action of $M$ is called an $M$-\emph{set}.
Suppose that $X$ and $X'$ are  $M$-sets.
One says that a map $\varphi \colon X \to X'$ is \emph{equivariant} if it satisfies
$\varphi(mx) = m\varphi(x)$ for all $m \in M$ and $x \in X$.

\subsection{Configuration spaces}
Let $M$ be a monoid, called the \emph{universe}, and let $A$ be a set, called the \emph{alphabet}.
Consider the set $A^M$ consisting of all maps $c \colon M \to A$.
An element $c \in A^M$ is called a \emph{configuration} over $(M,A)$.
\par
The map $M \times A^M \to A^M$, defined by $(m,c) \mapsto  c \circ R_m$ for all $(m,c) \in M \times A^M$, is an action of $M$ on $A^M$.
This action is called the $M$-\emph{shift}, or simply the \emph{shift}, on $A^M$.
\par
We equip $A^M$  with its \emph{prodiscrete uniform structure}, i.e., with  the product uniform structure on $A^M = \prod_{m \in M} A$ obtained by taking the discrete uniform structure on each factor $A$ of $A^M$
(see~\cite[Section~3]{csc-surj-monoids}).
The \emph{prodiscrete topology} on $A^M$ is the product topology on $A^M$ obtained by taking the discrete topology on each factor $A$ of $A^M$.
The prodiscrete topology on $A^M$ is also the topology associated with its prodiscrete uniform structure.
The $M$-shift action of $M$ on $A^M$ is uniformly continuous and hence continuous.  
In other words, for every $m \in M$, the map $A^M \to A^M$ given by $c \mapsto c \circ R_m$ is uniformly continuous and hence continuous.

\subsection{Cellular automata over monoids}
Let $M$ be a monoid and let $A$ be a set.
A map $\tau \colon A^M \to A^M$ is called a \emph{cellular automaton} over $(M,A)$ if
there exists a finite subset $S \subset M$ and a map $\mu \colon A^S \to A$
such that
\begin{equation}
\label{e:def-ca}
\tau(c)(m) = \mu((c \circ R_m)|_S)
\end{equation}
for all $c \in A^M$ and $m \in M$ (cf.~\cite{csc-surj-monoids}).
\par
If a finite subset $S \subset M$ and a map $\mu \colon A^S \to A$ are as in~\eqref{e:def-ca}, one says that $S $ is a \emph{memory set} for $\tau$
and that $\mu$ is the associated \emph{local defining map}. Note that $\mu$ is indeed entirely determined by $\tau$ and $S$,
since the restriction map $A^M \to A^S$ is surjective and $\mu(c|_S) = \mu((c \circ R_{1_M})|_S) = \tau(c)(1_M)$
for all $c \in A^M$ by~\eqref{e:def-ca}.

\begin{example}
\label{exa:left-m-is-ca}
Fix an element $m_0 \in M$.
Then the map $\tau  \colon A^M \to A^M$, defined by
  $\tau(c) \coloneqq c \circ L_{m_0}$ for all
$c \in A^M$,     is a cellular automaton.
The set $S \coloneqq  \{m_0\}$ is a memory set for $\tau$ and the identity map $\mu \coloneqq \Id_A \colon A^S =  A \to A$ is the associated local defining map.
Note that if $m_0 \coloneqq  1_M$ then   $\tau $ is the identity map $\Id_{A^M}$ on $A^M$.
\end{example}

If $\tau \colon A^M \to A^M$ is a cellular automaton, then a memory set for $\tau$ is not unique in general.
Indeed, if $S$ is a memory set for $\tau$ then any finite subset $S' \subset M$ such that $S \subset S'$ is also a memory set for $\tau$.
However, any cellular automaton $\tau \colon A^M \to A^M$ admits a unique memory set with minimal cardinality.
This memory set is called the \emph{minimal memory set} of $\tau$ and a finite subset of $M$ is a memory set for $\tau$ if and only if it contains the minimal memory set of $\tau$
(cf.~\cite[Proposition~4.13]{csc-surj-monoids}).  
\par
The \emph{generalized Curtis-Hedlund-Lyndon theorem}  \cite[Theorem~4.6]{csc-surj-monoids} asserts that a map $\tau \colon A^M \to A^M$ is a cellular automaton if and only if
it is uniformly continuous (with respect to the prodiscrete uniform structure on $A^M$) and equivariant (with respect to the shift action of $M$ on $A^M$).
In the particular case when $A$ is a finite set, the space $A^M$ is compact Hausdorff by Tychonoff's theorem so that the generalized Curtis-Hedlund-Lyndon theorem reduces to the assertion that 
$\tau$ is a cellular automaton if and only if it is continuous (with respect to the prodiscrete topology on $A^M$) and equivariant (with respect to the shift action of $M$ on $A^M$).
As the composite of two uniformly continuous (resp.\ equivariant) maps is uniformly continuous (resp.\ equivariant), one immediately deduces 
from the generalized Curtis-Hedlund-Lyndon theorem that if $\tau \colon A^M \to A^M$ and
$\sigma \colon A^M \to A^M$ are two cellular automata, then their composite $\sigma \circ \tau \colon A^M \to A^M$
is also a cellular automaton. As a consequence, the set $\CA(M,A)$ consisting
of all cellular automata over $(M,A)$ is a monoid whose identity element is  $\Id_{A^M}$ (cf.~Example~\ref{exa:left-m-is-ca}).

\subsection{Surjunctive monoids}
Recall that a monoid $M$ is said to be \emph{surjunctive} if, for any finite set $A$,
every injective cellular automaton $\tau \colon A^M \to A^M$ is surjective.

\begin{lemma}
\label{l:char-inj-ca}
Let $M$ be a monoid and let $A$ be a finite set.
Suppose that $\tau \colon A^M \to A^M$ is an injective cellular automaton.
Then there exists a cellular automaton $\sigma \colon A^M \to A^M$ such that
$\sigma \circ \tau = \Id_{A^M}$.
In other words, $\tau$ is left-invertible in the monoid $\CA(M,A)$.
\end{lemma}

\begin{proof}
Let us set $Y \coloneqq \tau(A^M)$. Consider the bijective map $\varphi \colon A^M \to Y$ given by $\varphi(c) \coloneqq \tau(c) $ for all $c \in A^M$.
It follows from the compactness of $A^M$ and the continuity of $\tau$ that
$\varphi$ is a homeomorphism.
Let $\psi \coloneqq \varphi^{-1} \colon Y \to A^M$ denote the inverse homeomorphism.
Since the map from $Y$ to the finite discrete set $A$, given by $y \mapsto \psi(y)(1_M)$ for all $y \in Y$, is continuous, there exists a finite subset $S \subset M$ and a map $\mu \colon A^S \to A$ such that
$\psi(y)(1_M) = \mu(y\vert_S)$ for all $y \in Y$.
Observe now that $\psi$ is $M$-equivariant since it is the inverse of an $M$-equivariant bijective map.
Thus, for all $y \in Y$ and $m \in M$, we have
\[
\psi(y)(m) = (\psi(y) \circ R_m)(1_M) = \psi(y \circ R_m)(1_M) = \mu((y \circ R_m)|_S).  
\]
We deduce that $\psi$ is the restriction to $Y$ of the cellular automaton $\sigma \colon A^M \to A^M$ admitting $S$ as a memory set and $\mu$ as the associated local defining map.
For all $c \in A^M$, we then have
\[
(\sigma \circ \tau)(c) = \sigma(\tau(c)) = \psi(\varphi(c)) = c.
\]
This shows that $\sigma \circ \tau = \Id_{A^M}$.
\end{proof}

In the group setting, the following result is \cite[Theorem~1]{castillo-ramirez_ca}.

\begin{proposition}
\label{p:char-surj-monoids}
Let $M$ be a monoid.
Then the following conditions are equivalent.
\begin{enumerate}[\rm (a)]
\item The monoid $M$ is surjunctive;
\item for every finite set $A$, the monoid $\CA(M,A)$ is directly finite.
\end{enumerate}
\end{proposition}

\begin{proof}
Suppose (a).
Let $A$ be a finite set and let $\sigma, \tau \in \CA(M,A)$ such that $\sigma \circ \tau = \Id_{A^M}$.
As $\Id_{A^M}$ is injective, this implies that $\tau$ is injective.
Thus, $\tau$ is bijective by (a).
From $\sigma \circ \tau = \Id_{A^M}$, we then deduce that $\sigma = \tau^{-1}$, so that 
$\tau \circ \sigma = \Id_{A^M}$.
This shows that $\CA(M,A)$ is directly finite.
Therefore, (a) implies (b).
\par
Conversely, suppose (b). Let $A$ be a finite set and let   $\tau \colon A^M \to A^M$ be an injective cellular automaton.
 By Lemma~\ref{l:char-inj-ca}, $\tau$ is a left-invertible element in $\CA(M,A)$.
Since $\CA(M,A)$ is directly finite by (b), we deduce that $\tau$ is invertible in $\CA(M,A)$.
In particular, $\tau$ is bijective and hence surjective.
This shows that (b) implies (a).
\end{proof} 
  
\subsection{Model theory of algebraically closed fields}
We refer to the monograph by Marker~\cite{marker-model-theory-gtm} for a detailed exposition of first-order logic, model theory,
and in particular the results we shall use concerning elementary equivalence of algebraically closed fields 
(see also \cite[Section~5]{gromov-esav}).
\par
The \emph{first-order language of rings} is the tuple ${\mathcal L} = (+,-,\times, 0,1)$, where
$+$ and $\times$ are \emph{binary function symbols} (corresponding to the ``sum'' and ``multiplication'' operations, respectively),
$-$ is a \emph{unary function symbol} (corresponding to the ``opposite''), and $0$ and $1$ are \emph{constant symbols}
(corresponding to the ``zero'' and the ``unity'', respectively).
\par
By recursively composing function symbols (resp.\ relation symbols) and using the \emph{quantifiers} $\forall$ and $\exists$ for \emph{variables} $x,y,z \ldots$, as well as the \emph{logical connectives} $\neg$ (\emph{negation}), $\land$ (\emph{conjunction}), $\lor$ (\emph{disjunction}), and $\rightarrow$ (\emph{implication}), one forms (first-order) \emph{formulas}. For example, a formula
expressing direct finiteness is given by
\[
\forall x \forall y: xy = 1 \rightarrow yx = 1.
\]
A formula with no free variables is called a \emph{sentence}.
\par
Two fields are called \emph{elementary equivalent} if they satisfy the same sentences in the language of rings.
The following  results play a pivotal role in our proof of Theorem~\ref{t:main}.

\begin{theorem}[First Lefschetz principle]
\label{t:el-equiv-fields}
Any two algebraically closed fields of the same characteristic are elementary equivalent.
\end{theorem}

\begin{proof}
This is a classical result in model theory which can be rephrased by saying that the theory of algebraically closed fields of a fixed characteristic is complete 
(see  Proposition~2.2.5 and Theorem~2.2.6 in~\cite{marker-model-theory-gtm}).
\end{proof}

Isomorphic fields are always elementary equivalent but the converse does not hold in general.
For example, the algebraic closure $\overline{\Q}$ of $\Q$ and the field $\C$ of complex numbers,
being two algebraically closed fields of  characteristic $0$,
are elementary equivalent by Theorem \ref{t:el-equiv-fields}, but are not isomorphic
since $\overline{\Q}$ is countable while $\C$ is uncountable.

\begin{theorem}[Second Lefschetz principle]
\label{t:char-p-char-0}
Let $\psi$ be a sentence in the language of rings which is satisfied by some (and therefore any) algebraically closed field of characteristic $0$.
Then there exists an integer $N$ such that $\psi$ is satisfied by any algebraically closed field of characteristic $p \geq N$.
\end{theorem}

\begin{proof}
This is (iii) $\implies$ (v) in \cite[Corollary~2.2.10]{marker-model-theory-gtm}.
\end{proof}

% SECTION 3
\section{Linear cellular automata over monoids}
\label{s:lca}

\subsection{Definition}
Let $M$ be a monoid and let $V$ be a vector space over a field $K$.
Then the  configuration space $V^M  = \prod_{m \in M} V$ over the alphabet $V$ inherits from $V$ a natural $K$-vector space product structure
for which
\[
(\lambda c + \lambda' c')(m) = \lambda c(m) + \lambda' c'(m)
\]
for all $\lambda,\lambda' \in K$, $c,c' \in V^M$, and $m \in M$.
Note that the shift action of $M$ on $V^M$ is linear, that is, for every $m \in M$, the map
$V^M \to V^M$, given by $c \mapsto c \circ R_m$, is $K$-linear.
\par
A cellular automaton $\tau \colon V^M \to V^M$ over $(M,V)$ which is $K$-linear is called a \emph{linear cellular automaton}.

\begin{proposition}
\label{p:char-lca-ldm}
Let $M$ be a monoid and let $V$ be a vector space over a field $K$.
Let $\tau \colon V^M \to V^M$ be a cellular automaton. 
Suppose that  $S$ is a memory set for $\tau$ and denote by $\mu \colon V^S \to V$
the associated local defining map. Then the following conditions are equivalent:
\begin{enumerate}[\rm (a)]
\item $\tau$ is a linear cellular automaton;
\item $\mu$ is $K$-linear.
\end{enumerate}
\end{proposition}

\begin{proof}
For $T \subset M$, denote by $\pi_T \colon V^M \to V^T$ the projection map
and write $\pi_m \coloneqq \pi_{\{m\}}$ for every $m \in M$. 
\par
Suppose first that $\tau$ is $K$-linear.
Consider the embedding map $\psi \colon V^S \to V^M$ defined, for every $p \in V^S$, by   $\psi(p)(m) \coloneqq  p(m)$ if $m \in S$ and $\psi(p)(m) \coloneqq 0_V$ otherwise.
It follows from~\eqref{e:def-ca} that $\mu = \pi_{1_M} \circ \tau \circ \psi$.
As $\pi_{1_M}$ and $\psi$ are clearly $K$-linear, we deduce that $\mu$ is $K$-linear. This shows that (a) implies (b).
\par
Suppose now that $\mu$ is $K$-linear.
It follows from~\eqref{e:def-ca} that
$\tau(c) = (\mu \circ \pi_S (c \circ R_m))_{m \in M}$ for every $c \in V^M$.
As $\pi_S$ and the action of $M$ on $V^M$ are $K$-linear, we deduce that $\tau$ is $K$-linear.
This shows that (b) implies (a).
 \end{proof}

Given a monoid $M$ and a vector space $V$ over a field $K$, we denote by $\LCA(M,V)$ the set of all linear cellular automata $\tau \colon V^M \to V^M$.
Consider the $K$-algebra $\End_K(V^M)$
consisting of all endomorphisms of the $K$-vector space $V^M$.
From the generalized Curtis-Hedlund-Lyndon theorem,
we deduce that $\LCA(V,M)$ is a subalgebra of  $\End_K(V^M)$. 

\subsection{Representation of one-dimensional linear cellular automata}
Let $M$ be a monoid and let $K$ be a field.
Recall from the Introduction that $K[M]$ is the $K$-algebra admitting $M$ as a basis for its underlying $K$-vector space structure and whose multiplication is obtained by extending $K$-linearly the monoid operation on $M$.
Thus, every $\alpha \in K[M]$ can be uniquely written in the form
\begin{equation}
\label{e:element-K-M}
\alpha = \sum_{m \in M} \alpha_m m,
\end{equation} 
with $\alpha_m \in K$ for all $m \in M$ and $\alpha_m = 0$ for all but finitely $m \in M$.
The finite set $\supp(\alpha)  \coloneqq  \{m \in M : \alpha_m \not= 0\}$ is called the \emph{support} of $\alpha$.
\par
If $\alpha = \sum_{m \in M} \alpha_m m, \alpha' = \sum_{m \in M} \alpha'_m m \in K[M]$ and $\lambda, \lambda' \in K$,  we have
\[
\lambda \alpha + \lambda' \alpha' = \sum_{m \in M} (\lambda \alpha_m + \lambda' \alpha'_m) m
\]
 and
 \[
 \alpha \alpha' = \sum_{m,m' \in M} \alpha_m \alpha'_{m'} m m'.
 \] 
Let $c \in K^M$ and let $\alpha = \sum_{m \in M} \alpha_m m \in K[M]$.  
Let $S \coloneqq \supp(\alpha) \subset M$ denote the support of $\alpha$.
We define the configuration $c * \alpha  \in K^M$  by setting
\begin{equation}
\label{e:conv-K-K}
(c * \alpha)(m) \coloneqq \sum_{s \in S} c(s m) \alpha_s 
\end{equation}
for all $m \in M$.
\par
Observe that the map $K^M \times K[M] \to K^M$, given by $(c,\alpha) \mapsto c * \alpha$ is $K$-bilinear.

\begin{proposition}
\label{p:alpha-beta-c}
Let $M$ be a monoid and let $K$ be a field.
Let $c \in K^M$ and let $\alpha,\beta \in K[M]$.
Then one has $(c * \alpha) *\beta = c * (\alpha \beta) $.
\end{proposition}

\begin{proof}
Write $\alpha = \sum_{m \in M} \alpha_m m$ and $\beta = \sum_{m \in M} \beta_m m$ as in~\eqref{e:element-K-M}.
Let $S \coloneqq \supp(\alpha)$ and $T \coloneqq \supp(\beta)$.
Observe that $\supp(\alpha \beta) \subset ST$.
For all $m \in M$, we have
\begin{align*}
((c * \alpha) * \beta )(m)
&= \sum_{t \in T} (c * \alpha)(t m) \beta_t \\
&= \sum_{t \in T}  \left(\sum_{s \in S}  c(s t m) \alpha_s \right) \beta_t \\
&= \sum_{s \in S, t \in T} c(s t m) \alpha_s \beta_t  \\
&= \sum_{s \in S, t \in T : s t = u} c(u m) \alpha_s \beta_t  \\
&= \sum_{ u \in ST}  c(u m) (\alpha \beta)_u\\
&= (c * (\alpha\beta)) (m).
\end{align*}
This shows that  $(c * \alpha)  * \beta = c * (\alpha \beta)$.
\end{proof}

\begin{proposition}
\label{p:rep-1-dim-lca}
Let $M$ be a  monoid and let $K$ be a field.
Let $\alpha = \sum_{m \in M} \alpha_m m \in K[M]$  and let $S \coloneqq \supp(\alpha)$. 
Define the map 
$\tau_\alpha \colon K^M \to K^M$ by
\begin{equation}
\label{e:def-tau-alpha-1}
\tau_\alpha(c) \coloneqq c * \alpha
\end{equation}
for all $c \in K^M$.
Then $\tau_\alpha$ is a linear cellular automaton whose minimal memory set is $S$ and whose associated local defining map  is the map $\mu \colon K^S \to K$ defined by
\begin{equation}
\label{e:mu-alpha}
\mu(p) \coloneqq  \sum_{s \in S} p(s)\alpha_s
\end{equation}
for all $p \in  K^S$.
\end{proposition}

\begin{proof}
For all $c \in K^M$ and $m \in M$, we have
\[
\tau_\alpha(c)(m) = (c * \alpha )(m) = \sum_{s \in S} c(s m) \alpha_s = \sum_{s \in S} (c \circ R_m)(s)\alpha_s.
\]
We deduce that
\[
\tau_\alpha(c)(m) = \mu((c \circ R_m)|_S),
\]
where $\mu \colon K^S \to K$ is as in~\eqref{e:mu-alpha}.
This shows that $\tau_\alpha$ is a cellular automaton over $(M,K)$ admitting $S$ as a memory set with associated local defining map $\mu$.
As $\mu$ is $K$-linear, it follows from Proposition~\ref{p:char-lca-ldm} that $\tau_\alpha$ is a linear cellular automaton.
\par
It remains to show that $S$ is the minimal memory set for $\tau_\alpha$.
Suppose by contradiction that there exists $s_0 \in S$ such that $T \coloneqq S \setminus \{s_0\}$ is a memory set for $\tau_\alpha$.
Consider the configuration $c \in K^M$ defined by $c(s_0) \coloneqq  1$ and $c(m) \coloneqq 0$ for all $m \in M \setminus \{s_0\}$.
Let $\nu \colon K^T \to K$ denote the local defining map for $\tau_\alpha$ associated with $T$.
We then have
\[
\tau_\alpha(c)(1_M)  = \nu(c|_T) = \nu(0_{K^T}) = 0
\]
since $\nu$ is $K$-linear by Proposition~\ref{p:char-lca-ldm}.
On the other hand, we have 
\[
\tau_\alpha(c)(1_M) = (c * \alpha)(1_M) = \sum_{s \in S} c(s) \alpha_s = \alpha_{s_0} \not= 0
\]
since $s_0 \in S = \supp(\alpha)$, a contradiction.
This shows that  $S$ is the minimal memory set for $\tau_\alpha$.
\end{proof}

\begin{theorem}
\label{t:rep-one-dim-lca}
Let $M$ be a  monoid and let $K$ be a field.
Then the map $\Phi \colon K[M] \to \LCA(M,K)$ given by
$\Phi(\alpha) \coloneqq \tau_\alpha$, where $\tau_\alpha$ is defined by~\eqref{e:def-tau-alpha-1},  is an anti-isomorphism of $K$-algebras.
\end{theorem}

\begin{proof}
Let $\lambda,\lambda' \in K$ and $\alpha,\alpha' \in K[M]$.
For all $c \in K^M$, we have
\begin{align*}
\Phi(\lambda \alpha + \lambda' \alpha')(c) &= \tau_{\lambda \alpha + \lambda' \alpha'}(c) \\
&= c * (\lambda \alpha + \lambda' \alpha')  \\
&= \lambda c * \alpha  + \lambda' c * \alpha' \\
&= \lambda \tau_\alpha(c) + \lambda' \tau_{\alpha'}(c) \\
&= \lambda \Phi(\alpha)(c) + \lambda' \Phi(\alpha')(c) \\
&= (\lambda \Phi(\alpha) + \lambda' \Phi(\alpha'))(c).
\end{align*}
It follows that $\Phi(\lambda \alpha + \lambda' \alpha') = \lambda \Phi(\alpha) + \lambda' \Phi(\alpha')$.
 This shows that  $\Phi$ is $K$-linear.
 \par
The minimal memory set of $0 \in \LCA(M,K)$ is the empty set. Since, for every $\alpha \in K[M]$,  the minimal memory set of $\Phi(\alpha) = \tau_\alpha$ is the support of $\alpha$ by Proposition~\ref{p:rep-1-dim-lca},
 we deduce that $\Phi$ is injective.  
 \par
For all $c \in K^M$ and $m \in M$, we have $\tau_{1_M}(c)(m)  = c(m)$.
Therefore, $\Phi(1_M)$ is the identity map on $K^M$, i.e., the multiplicative identity element of  $\LCA(M,K)$. 
\par
Suppose that $\sigma \colon K^M \to K^M$ is a linear cellular automaton with memory set $S$ and associated local defining map $\mu \colon K^S \to K$.
Since $\mu$ is $K$-linear by Proposition~\ref{p:char-lca-ldm},
there exist $\beta_s \in K$, $s \in S$, such that
$\mu(p) = \sum_{s \in S} p(s)\beta_s$ for all $p \in K^S$.
Up to replacing $S$ by a smaller subset, we can assume $\beta_s \not= 0$ for all $s \in S$.
Setting $\beta \coloneqq \sum_{s \in S}\beta_s s \in K[M]$, we then have $\Phi(\beta) = \tau_\beta = \sigma$ by Proposition~\ref{p:rep-1-dim-lca}.
This shows that $\Phi$ is surjective. 
 \par
Let $\alpha,\beta \in K[M]$. For all $c \in K^m$, we have
\begin{align*}
\tau_{\alpha\beta}(c) &= c * (\alpha\beta)  \\
&= (c * \alpha) * \beta ) &&\text{(by Proposition~\ref{p:alpha-beta-c})} \\
&= \tau_\alpha(c) * \beta \\
&= \tau_\beta(\tau_\alpha(c)) \\
&= (\tau_\beta \circ \tau_\alpha)(c).
\end{align*}
Therefore, we have $\Phi(\alpha\beta) = \tau_{\alpha\beta} = \tau_\beta \circ \tau_\alpha =  \Phi(\beta) \circ \Phi(\alpha)$.
This completes the proof that $\Phi$ is an anti-isomorphism of $K$-algebras.
\end{proof}

\subsection{Matrix representation of finite-dimensional linear cellular automata}
Let $M$ be a monoid, let $K$ be a field, and let $d \geq 1$ be an integer.
In the sequel, we shall systematically use the natural identification $(K^d)^M = (K^M)^d$.
\par 
Given a configuration
\begin{equation}
\label{e:c-row-vector-in-K-M}
c = (c_1,\dots,c_d) \in (K^M)^d = (K^d)^M
\end{equation}
and a matrix
\[ 
A = (\alpha_{i j})_{1 \leq i,j \leq d}
= 
\begin{pmatrix}
\alpha_{1 1} &\dots & \alpha_{1 d} \\
\vdots & \vdots & \vdots \\
\alpha_{d 1} & \dots & \alpha_{d d}
\end{pmatrix} 
\in \Mat_d(K[M]),
\] 
we define $c * A \in (K^d)^M = (K^M)^d$ by
\begin{equation}
\label{e:A-star-c-col-matrix-K-M}
c * A = ((c * A)_1, \dots, (c * A)_d) 
\quad
\text{with}
\quad
(c * A )_j \coloneqq  \sum_{i = 1}^d c_i * \alpha_{i j}  \in K^M \text{  for  } 1 \leq j \leq d.
\end{equation}
Observe that this extends the definition of the operation $*$ defined in Formula~\eqref{e:conv-K-K}.
Note also that the map $(K^d)^M \times \Mat_d(K[M])  \to (K^d)^M$, given by $(c,A) \mapsto c * A$, is $K$-bilinear.
\par
The following result extends Proposition~\ref{p:alpha-beta-c}.

\begin{proposition}
\label{p:A-B-c}
Let $M$ be a monoid, let $K$ be a field, and let $d \geq 1$ be an integer.
Let $c \in (K^d)^M$ and let $A,B \in \Mat_d(K[M])$.  
Then one has $(c * A)* B   = c * (AB)$.
\end{proposition}

\begin{proof}
Let us write $A = (\alpha_{i j})_{1 \leq i,j \leq d}$, $B = (\beta_{i j})_{1 \leq i,j \leq d}$, and $C \coloneqq  A B  = (\gamma_{i j})_{1 \leq i,j \leq d}$.
For all $1 \leq j \leq d$, we have
\begin{align*}
((c * A) * B )_j
&= \sum_{i=1}^d (c * A)_i * \beta_{i j} \\
&= \sum_{i=1}^d \left(\sum_{k =1}^d c_k * \alpha_{k i}\right) * \beta_{i j}\\
&= \sum_{i=1}^d\sum_{k=1}^d (c_k * \alpha_{k i}) * \beta_{i j} \\
&= \sum_{i=1}^d\sum_{k=1}^d c_k * (\alpha_{k i}  \beta_{i j}) && \mbox{(by Proposition \ref{p:alpha-beta-c})} \\
&= \sum_{k=1}^d c_k * \left(\sum_{i=1}^d \alpha_{k i} \beta_{i j}\right)\\
&= \sum_{k=1}^d c_k * \gamma_{k j} \\
& = (c * C)_j.
\end{align*}
This shows that $(c * A)* B = c * (AB)$.
\end{proof}

\begin{remark}
\label{rem:KdM-right-module}
Denoting by $I_d$ the identity matrix in $\Mat_d(K[M])$,
it follows from~\eqref{e:A-star-c-col-matrix-K-M} that $c * I_d  = c$ for all $c \in (K^d)^M$.
As $(c + c') * A = c * A + c' * A$ (by linearity of $*$ with respect to the first variable) and $(c * A) * B = c* (AB)$ (by Proposition~\ref{p:A-B-c}) for all $c,c' \in (K^d)^M$ and $A,B \in \Mat_d(\K[M])$, 
we deduce that the abelian group $(K^d)^M$ equipped with the scalar multiplication associated with $*$, has the structure of a right $\Mat_d(K[M])$-module.
\end{remark}

In the sequel, we shall use the following notation. Define the \emph{support} of a matrix $A = (\alpha_{i j})_{1 \leq i,j \leq d}  \in \Mat_d(K[M])$ as  the finite subset
$\supp(A) \subset M$ given by
\[
\supp(A) \coloneqq \bigcup_{1 \leq i,j \leq d} \supp(\alpha_{i j}).
\]
Thus, writing 
\[
\alpha_{i j} = \sum_{m \in M} \alpha_{i,j,m} m
\]
with $\alpha_{i,j,m} \in K$ for all $m \in M$ and $1 \leq i,j \leq d$ (cf.~\eqref{e:element-K-M}), we have
\[
\supp(A) = \{m \in M : \alpha_{i,j,m} \not= 0 \text{ for some } 1 \leq i,j \leq d\}.
\] 
 
The following result extends Proposition~\ref{p:rep-1-dim-lca}.

\begin{proposition}
\label{p:rep-d-dim-lca}
Let $M$ be a monoid, let $K$ be a field, and let $d \geq 1$ be an integer.
Let $A = (\alpha_{i j})_{1 \leq i,j \leq d}   \in \Mat_d(K[M])$ and let $S \coloneqq \supp(A)$.
Define the map 
$\tau_A \colon (K^d)^M \to (K^d)^M$ by
\begin{equation}
\label{e:def-tau-A-d}
\tau_A(c) \coloneqq c * A
\end{equation}
for all $c \in (K^d)^M$, where $c * A$ is given by~\eqref{e:A-star-c-col-matrix-K-M}.
Then $\tau_A$ is a linear cellular automaton whose minimal memory set is $S$ and whose associated local defining map  is the map $\mu \colon (K^d)^S  \to K^d$ defined by
\begin{equation}
\label{e:mu-matrice}
\mu(p) = (\mu_1(p),\dots,\mu_d(p))
\quad
\text{with}
\quad
\mu_j(p) \coloneqq  \sum_{i = 1}^d p_i(s)\alpha_{i, j, s}
\end{equation}
for all $p = (p_1,\dots,p_d) \in   (K^S)^d = (K^d)^S$.
\end{proposition}

\begin{proof}
Given $i,j \in \{1,\dots,d\}$, we have  $\supp(\alpha_{i j}) \subset S$. Thus, it follows from Proposition~\ref{p:rep-1-dim-lca} that the map $K^M \to K^M$, given by $c_i \mapsto c_i * \alpha_{i j} $,
is a linear cellular automaton admitting $S$ as a memory set with associated local defining map
$\mu_{i j} \colon K^S \to K$ given by
\[
\mu_{i j}(p_i) \coloneqq  \sum_{s \in S} p_i(s)\alpha_{i,j,s}
\]
 for all $p_i \in K^S$.
 Consequently, we have
 \[
 (c_i * \alpha_{i j} )(m) = \mu_{i j}((c_i \circ R_m)|_S)
 \]
 for all $c_i \in K^M$.
 By linearity, this gives us
\begin{align*}
(c * A)_j(m) &=  \sum_{i = 1}^d (c_i * \alpha_{i j} )(m) \\
&= \sum_{i = 1}^d \mu_{i j}((c_i \circ R_m)|_S) \\
&= \mu_j((c \circ R_m)|_S) && \mbox{(cf.\ \eqref{e:mu-matrice})}
\end{align*}
 for all $c = (c_1,\dots,c_d) \in (K^M)^d = (K^d)^M$.
 Since the map $\mu = (\mu_1,\dots,\mu_d) \colon (K^S)^d = (K^d)^S \to K^d$ is $K$-linear,
we deduce that $\tau_A \colon (K^d)^M \to (K^d)^M$ is a linear cellular automaton with memory set $S$ and associated local defining map $\mu$.
 \par
 It remains  to show that $S$ is the minimal memory set of $\tau_A$.
Suppose by contradiction that there exists $s_0 \in S$ such that $T \coloneqq S \setminus \{s_0\}$ is a memory set for $\tau_\alpha$.
As $s_0$ is in the support of $A$,
there exists  $i_0,j_0 \in \{1,\dots,d\}$ such that $\alpha_{i_0,j_0,s_0} \not= 0$.
  Consider the configuration $c = (c_1,\dots,c_d) \in (K^M)^d = (K^d)^M$ defined by 
  $c_{i_0}(s_0) \coloneqq  1$ and $c_i(m) \coloneqq 0$ for all $(i,m) \in \{1,\dots,d\} \times M \setminus \{(i_0,s_0)\}$.
Let $\nu \colon (K^d)^T \to K^d$ denote the local defining map for $\tau_A$ associated with $T$.
We then have
\[
\tau_A(c)(1_M)  = \nu(c|_T) = \nu(0_{(K^d)^T}) = 0_{K^d}
\]
since $\nu$ is $K$-linear by Proposition~\ref{p:char-lca-ldm}.
On the other hand, the $j_0$-th component of $\tau_A(c)(1_M)$ is 
\[
(c * A)_{j_0}(1_M) = \sum_{i=1}^d\sum_{s \in S} \alpha_{i,j_0,s} c_i(s) = \alpha_{i_0,j_0,s_0} \not= 0
\]
since $s_0 \in \supp(\alpha_{i_0 j_0})$, a contradiction.
This shows that  $S$ is the minimal memory set for $\tau_A$.
\end{proof}

The following result extends Theorem~\ref{t:rep-one-dim-lca}.

\begin{theorem}
\label{t:rep-d-dim-lca}
Let $M$ be a monoid, let $K$ be a field, and let $d \geq 1$ be an integer.
Then the map $\Psi \colon \Mat_d(K[M]) \to \LCA(M,K^d)$ given by $\Psi(A) \coloneqq \tau_A$
for all $A \in \Mat_d(K[M])$, where $\tau_A$ is defined by~\eqref{e:def-tau-A-d},  is an anti-isomorphism of $K$-algebras.
\end{theorem}

\begin{proof}
Let $\lambda,\lambda' \in K$ and $A,A' \in \Mat_d(K[M])$.
For all $c  \in  (K^d)^M$, we have
\begin{align*}
\Psi(\lambda A + \lambda' A')(c) 
&= \tau_{\lambda A + \lambda' A'}(c) \\
&= c * (\lambda A + \lambda' A') \\
&= \lambda (c * A)  + \lambda' (c *  A') \\
&= \lambda \tau_A(c) + \lambda' \tau_{A'}(c) \\
&= \lambda \Psi(A)(c) + \lambda' \Psi(A')(c) \\
&= (\lambda \Phi(A) + \lambda' \Psi(A'))(c).
\end{align*}
It follows that $\Psi(\lambda A + \lambda' A') = \lambda \Psi(A) + \lambda' \Psi(A')$.
 Therefore, the map  $\Psi$ is $K$-linear.
 \par
The minimal memory set of $0 \in \LCA(M,K^d)$ is the empty set. 
Since, for every $A \in \Mat_d(K[M])$,  the minimal memory set of $\Psi(A)$ is the support of $A$ by Proposition~\ref{p:rep-d-dim-lca},
 we deduce that $\Psi$ is injective.  
 \par
 Denoting by $I_d$ the identity matrix in $\Mat_d(K[M])$,
we have  $c * I_d  = c$ for all $c \in (K^d)^M$ (cf.~Remark~\ref{rem:KdM-right-module}).
 Therefore, $\Psi(I_d) = \tau_{I_d}$ is the identity map on $(K^d)^M$, i.e., the multiplicative identity element of  $\LCA(M,K^d)$. 
\par
Let us show that $\Psi$ is surjective.
Suppose that $\tau \in \LCA(M,K^d)$. 
For  $i,j \in \{1,\dots,d\}$, consider the embedding $\iota_i \colon K^M \to (K^M)^d = (K^d)^M$ 
sending every $c_i \in K^M$ to 
the configuration $(0,\dots,0,c_i,0,\dots,0) \in (K^M)^d = ((K^d)^M$, where $c_i$ is located on the $i$-th component, 
and the projection  map $\pi_j \colon  (K^d) ^M = (K^M)^d  \to K^M$
onto the $j$-th factor of $(K^M)^d$.
The maps $\iota_i$ and $\pi_j$ are clearly $K$-linear, uniformly continuous (with respect to the prodiscrete uniform structures on $K^M$ and $(K^d)^M$), and equivariant (with respect to the shift actions of $M$ on $K^M$ and $(K^d)^M$).
By using the generalized Curtis-Hedlund-Lyndon theorem,
we deduce that the composite map $\tau_{i j} \coloneqq \pi_j \circ \tau \circ \iota_i \colon K^M \to K^M$ is 
a linear cellular automaton over $(M,K^d)$.
It then follows from the surjectivity of $\Phi$ in Theorem~\ref{t:rep-one-dim-lca}
that there exists $\alpha_{i j} \in K[M]$ such that
$\tau_{i j} (c_i) = c_i * \alpha_{i j} $ for all $c_i \in K^M$.
By linearity, setting $A \coloneqq (\alpha_{i j})_{1 \leq i,j \leq d} \in \Mat_d(K[M])$,
we get, for all $c = (c_1,\dots,c_d) \in (K^M)^d$,
\begin{align*}
\tau(c)
&= \tau(c_1,\dots,c_d) \\
&= \sum_{i = 1}^d \tau(\iota_i(c_i)) \\
&= \left(\sum_{i = 1}^d (\pi_1 \circ \tau \circ \iota_i)(c_i), \dots,\sum_{i = 1}^d (\pi_d \circ \tau \circ \iota_i)(c_i)\right) \\
&= \left(\sum_{i = 1}^d c_i * \alpha_{i 1} ,\dots, \sum_{i = 1}^d c_i * \alpha_{i d} \right) \\
&= \left((c * A)_1,\dots,(c * A )_d \right) \\
&= c * A\\
&= \tau_A(c).
\end{align*}
Therefore $\tau = \tau_A = \Psi(A)$.
This shows that $\Psi$ is surjective. 
\par
Let now $A,B \in \Mat_d(K[M])$. 
For all $c \in (K^d)^M$, we have
\begin{align*}
\tau_{AB}(c) 
&= c * (AB)  \\
&= (c * A)  * B &&\text{(by Proposition~\ref{p:A-B-c})} \\
&= \tau_B(\tau_A(c)) \\
&= (\tau_B \circ \tau_A)(c).
\end{align*}
It follows that  $\Psi(A B) = \tau_{AB} = \tau_B \circ \tau_A =  \Psi(B) \circ \Psi(A)$.
This completes the proof that $\Psi$ is an anti-isomorphism of $K$-algebras.
\end{proof}

As an immediate consequence of Theorem~\ref{t:rep-d-dim-lca}, we obtain the following result.

\begin{corollary}
\label{c:rep-d-dim-lca}
Let $M$ be a monoid and  let $K$ be a field.
Suppose that $V$ is a $K$-vector space with finite dimension $d \coloneqq \dim_K(V) \geq 1$.
Then the $K$-algebras $\LCA(M,V)$ and $\Mat_d(K[M])$ are anti-isomorphic. 
\end{corollary}

\begin{proof}
Choosing a basis for $V$, we get a $K$-vector space isomorphism $V \to K^d$.
This yields a $K$-algebra isomorphism $\Theta \colon \LCA(M,V) \to \LCA(M,K^d)$.
The composition $\Psi^{-1} \circ \Theta$ then provides the desired $K$-algebra
anti-isomorphism of $\LCA(M,V)$ onto $\Mat_d(K[M])$.
\end{proof}

\subsection{Linear surjunctivity and stable finiteness}

Let $M$ be a monoid. 
Given a field $K$, one says that $M$ is $K$-\emph{surjunctive} if, for
any finite dimensional vector space $V$ over $K$, every injective linear cellular automaton $\tau \colon V^M \to V^M$
is surjective. One says that $M$ is \emph{linearly surjunctive} 
if $M$ is $K$-surjunctive for every  field $K$.
Finally, one says that $M$ is \emph{finitely-linearly surjunctive}
if $M$ is $K$-surjunctive for every finite field $K$.
Observe that every surjunctive monoid is finitely linearly surjunctive since every finite-dimensional vector space over a finite field is finite.

\begin{theorem}
\label{t:lin-surj-d-dir-fin}
Let $M$ be a monoid,  let $K$ be a field, and let $d \geq 1$ be an integer.
Suppose that every injective linear cellular automaton $\tau \colon (K^d)^M \to (K^d)^M$ is surjective.
Then the $K$-algebra $\Mat_d(K[M])$ is directly finite. 
\end{theorem}

\begin{proof}
Suppose that $A,B \in \Mat_d(K[M])$ satisfy $AB = I_d \in \Mat_d(K[M])$.
Consider the $K$-algebra anti-isomorphism
$\Psi \colon \Mat_d(K[M]) \to \LCA(M,K^d)$, $A \mapsto \tau_A$, introduced in Theorem~\ref{t:rep-d-dim-lca}.
We then have
\[
\tau_B \circ \tau_A = \Psi(B) \circ \Psi(A) = \Psi(AB) = \Psi(I_d) = \Id_{(K^d)^M}.
\]
This implies that  the linear cellular automaton $\tau_A \colon (K^d)^M \to (K^d)^M$ is injective.
By our hypothesis, $\tau_A$ is surjective and therefore bijective.
As $\tau_B \circ \tau_A = \Id_{(K^d)^M}$, the inverse bijection  of $\tau_A$ is $\tau_B$.
We deduce that $\tau_A$ is an invertible element in $\LCA(M,K^d)$ with inverse  $\tau_B$. 
It follows that $A = \Psi^{-1}(\tau_A)$ is invertible in $\Mat_d(K[M])$ with inverse $B = \Psi^{-1}(\tau_B)$.
Therefore, $B A = I_d$.
This shows that $\Mat_d(K[M])$ is directly finite.
\end{proof}

The following statements are immediate consequences of Theorem~\ref{t:lin-surj-d-dir-fin}.

\begin{corollary}
\label{c:K-lin-surj-stable-fin}
Let $M$ be a monoid and let $K$ be a field.
Suppose that $M$ is $K$-surjunctive.
Then the monoid algebra $K[M]$ is stably finite.
\qed
\end{corollary}

\begin{corollary}
\label{c:lin-surj-stable-fin-1}
Let $M$ be a linearly-surjunctive monoid. 
Then the monoid algebra $K[M]$ is stably finite for every field $K$.
\qed
\end{corollary}

\begin{corollary}
\label{c:lin-surj-stable-fin}
Let $M$ be a finitely-linearly-surjunctive monoid. 
Then the monoid algebra $K[M]$ is stably finite for every finite field $K$.
\qed
\end{corollary}

\begin{corollary}
\label{c:lin-surj-stable-fin-2}
Let $M$ be a surjunctive monoid. 
Then the monoid algebra $K[M]$ is stably finite for every finite field $K$.
\qed
\end{corollary}

% SECTION 4
\section{Proof of the main result}
\label{s:proof-main-result}

In this section, we prove Theorem~\ref{t:main} and some other related results.

\begin{theorem}
\label{t:sf-KM-elem-equiv}
Let $M$ be a monoid,  let $d \geq 1$ be an integer,
and suppose that $K$ and $L$ are elementary equivalent fields.
Then  $\Mat_d(K[M])$ is directly finite if and only if  $\Mat_d(L[M])$ is directly finite.
\end{theorem}

\begin{proof}
We first  claim that, given a finite subset $S \subset M$,
there exists a sentence $\psi_{S}$ in the language of rings such that a field $K$ satisfies 
$\psi_{S}$ if and only if there exist two matrices $A,B \in \Mat_d(K[M])$ such that
\begin{enumerate}[\rm (1)]
\item
the support of each entry of $A$ and of each entry of $B$ is contained in $S$;
\item
$A B = I_d$ and $B A \not= I_d$.
\end{enumerate} 
Indeed, since $S$ is fixed, we can quantify over  matrices in $\Mat_d(K[M])$ whose support of each entry is contained in $S$ by quantifying over the coefficients of every entry of the matrix. Consequently, the existence of two matrices $A,B \in \Mat_d(K[M])$ satisfying (1) and  (2) can be expressed by a $2d^2|S|$-variables first-order sentence $\psi_{S}$ in the language of rings, depending
only on the multiplication table (in the monoid $M$) of the elements in $S$. 
\par 
For the sake of completeness, we give below an explicit formula for $\psi_{S}$. 
We represent the entries at the position $(i,j)$ of the matrices $A$ and $B$  by 
$\displaystyle \sum_{s \in S} x_{i, j,s} s$ and $\displaystyle \sum_{s \in S} y_{i, j,s} s$, respectively. 
For $ 1 \leq i,j \leq d$ and $m \in S^2 = \{st: s, t \in S\} \subset M$, set 
\[
P(X,Y)_{i,j,m}\coloneqq \sum_{k = 1}^d \sum_{\substack{s,t \,\in S\\ st= m}} x_{i,k,s} y_{k,j,t}.
\] 
Let $D \coloneqq \{(i,i,1_M) : 1 \leq i \leq d\}$. 
Then the property $AB=I_d$ (resp.\ $BA=I_d$) can be expressed by the first-order formula
$P(X,Y)$ (resp.\ $P(Y,X)$), where 
\[
P(X,Y)= \left(\bigwedge\limits_{\substack{ (i,i,1_M) \in D}} P(X,Y)_{i,i,1_M}=1 \right)
\land   \left(\bigwedge\limits_{\substack{m \in S^2, (i,j,m) \notin D}} P(X,Y)_{i,j,m}=0\right).
\]
Hence, we can take 
\begin{equation}
\label{e_psi-M}
\psi_{S} \coloneqq  \exists X = (x_{i,j,s})_{\substack{1 \leq i,j \leq d\\s \in S}}, \ \exists Y=(y_{i,j,s})_{\substack{1 \leq i,j \leq d\\s \in S}}: \ P(X,Y) \land \neg P(Y,X).
\end{equation} 
\par
As $K$ and $L$ play symmetric roles, it suffices to show that if $\Mat_d(K[M])$ is not directly finite then $\Mat_d(L[M])$ is not directly finite.
So, let us assume that $\Mat_d(K[M])$ is not directly finite.
This means that there exist  two square matrices $A$ and $B$ of order $d$ with entries in $K[M]$ such that
$A B = I_d$ and $B A \not= I_d$. If $S \subset M$ is a finite subset containing the support of each entry of $A$ and $B$,
the field $K$ satisfies the sentence $\psi_{S}$ given by the claim.
Since $K$ and $L$ are elementary equivalent by our hypothesis, the sentence $\psi_{S}$ is also satisfied by the field $L$.
We deduce that  $\Mat_d(L[M])$ is not directly finite either.
\end{proof}

As an immediate consequence of Theorem~\ref{t:sf-KM-elem-equiv}, we obtain the following result.

\begin{corollary}
\label{c:sf-KM-elem-equiv}
Let $M$ be a monoid and 
suppose that $K$ and $L$ are elementary equivalent fields.
Then  $K[M]$ is stably finite if and only if  $L[M]$ is stably finite.
\qed
\end{corollary}

\begin{theorem}
\label{t:df-L-finite-implies-df}
Let $M$ be a monoid and let $d \geq 1$ be an integer.
Suppose that $\Mat_d(L[M])$ is directly finite for every finite field $L$.
Then $\Mat_d(K[M])$ is directly finite for every  field $K$.
\end{theorem}

\begin{proof}
Let $K$ be a field.
We want to show that $\Mat_d(K[M])$ is directly finite.
We divide the proof into four cases.
\par
{\bf Case 1:} $K$ is the algebraic closure of the finite field $F_p \coloneqq \Z/p\Z$ for some prime $p$.
For every integer $n \geq 1$, let $K_n$ denote the subfield of $K$ consisting of all roots of the polynomial $X^{p^{n!}} - X$.
In other words, denoting by $\phi \colon K \to K$ the Frobenius automorphism, $K_n$ is the subfield of $K$ consisting of all 
fixed points of $\phi^{n!}$. 
We have $K_n \subset K_{n + 1}$ for all $n \geq 1$ and $K = \bigcup_{n \geq 1}  K_n$.  
Moreover,  $K_n$ is a finite field (of cardinality $p^{n!}$) for every $n \geq 1$.
Let $A$ and $B$ be square matrices of order $d$ with entries in $K[M]$ such that $A B = I_d$.
Then there exists $n_0 \geq 1$ such that all entries of $A$ and $B$ are in $K_{n_0}[M]$.
Since $\Mat_d(K_{n_0}[M])$ is directly finite by our hypothesis, we deduce that  $B A = I_d$.
This shows that $\Mat_d(K[M])$ is directly  finite.
\par
{\bf Case 2:} $K$ is an algebraically closed field of characteristic $p > 0$.
This follows from Case 1, Theorem~\ref{t:sf-KM-elem-equiv}, and the first Lefschetz principle (Theorem~\ref{t:el-equiv-fields}).
\par
{\bf Case 3:} $K$ is an  algebraically closed field of characteristic $0$.
Suppose by contradiction that $\Mat_d(K[M])$ is not directly finite.
This means that $K$ satisfies the first-order sentence $\psi_{S}$ in~\eqref{e_psi-M} for  some finite subset $S \subset M$.
By applying the second Lefschetz principle (Theorem~\ref{t:char-p-char-0}), we deduce that there exists an integer $N \geq 1$ such that
$\psi_{S}$ is satisfied by any algebraically closed field $L$ of characteristic $p \geq N$.
This implies that for such a field $L$, the $L$-algebra $\Mat_d(L[M])$ is not directly finite, in contradiction with Case 2.
\par
{\bf Case 4:} $K$ is an arbitrary field.
Let $\overline{K}$ denote the algebraic closure of $K$.
Then $\Mat_d(\overline{K}[M])$ is directly finite by Cases 2 and 3.
As $\Mat_d(K[M])$ is a subring of $\Mat_d(\overline{K}[M])$, we deduce that $\Mat_d(K[M])$ is itself stably finite.
\end{proof}

The following result is an immediate consequence of Theorem~\ref{t:df-L-finite-implies-df}.

\begin{corollary}
\label{c:df-L-finite-implies-df-1}
Let $M$ be a monoid.
Suppose that $L[M]$ is stably finite for every finite field $L$.
Then $K[M]$ is stably finite for every  field $K$.
\qed
\end{corollary}

\begin{theorem}
\label{t:fin-lin-implies-sf}
Let $M$ be a finitely-linearly surjunctive monoid and let $K$ be a field.
Then the monoid algebra $K[M]$ is stably finite.
\end{theorem}

\begin{proof}
Since $M$ is finitely-linearly surjunctive, $L[M]$ is stably finite for every finite field $L$
by Corollary~\ref{c:lin-surj-stable-fin}.
Therefore, $K[M]$ is stably finite  by Corollary~\ref{c:df-L-finite-implies-df-1}. 
\end{proof}

\begin{proof}[Proof of Theorem~\ref{t:main}]
Since every surjunctive monoid is finitely-linearly surjunctive,
Theorem~\ref{t:main} is a  consequence of Theorem~\ref{t:fin-lin-implies-sf}.
\end{proof}

  % SECTION 5
\section{Some open problems}

To the open questions related to surjunctivity of monoids listed at the end of~\cite{csc-surj-monoids}, we add the following ones. 

\begin{enumerate}[(Q1)]
% Q1
\item
Let $M$ be a cancellative monoid and let $K$ be a field.
Is $K[M]$ stably finite?
% Q2
\item
(cf.~\cite{steinberg2022stable})
Let $M$ be a monoid containing no submonoid isomorphic to the bicyclic monoid and let $K$ be a field.
Is $K[M]$ stably finite?
% Q3
\item
Do there exist a monoid $M$ and a field $K$ such that
$K[M]$ is directly finite but not stably finite?
% Q4
\item
Given an integer $d \geq 1$,
do there exist a monoid $M$ and a field $K$ such that
$\Mat_d(K[M])$ is directly finite but $\Mat_{d + 1}(K[M])$ is not?
\end{enumerate}

Observe that a positive answer to (Q2) would yield a positive answer to (Q1) and 
that a positive answer to (Q1) would yield  Kaplansky's stable finiteness conjecture for groups.
Note also that a positive answer to (Q4) would yield a positive answer to (Q3).

\bibliographystyle{siam}

\end{document}